\documentclass[12pt,reqno]{amsart}%
\usepackage{amsfonts,anysize}
\usepackage{txfonts,bm}
\usepackage[colorlinks,urlcolor=blue,citecolor=blue,linkcolor=blue]{hyperref}
\usepackage{amsmath}
\usepackage{amsfonts}
\usepackage{amssymb}
\marginsize{2cm}{2cm}{1cm}{1cm}

\newtheorem{thm}{Theorem}[section]
\newtheorem{cor}[thm]{Corollary}

\newtheorem{lem}[thm]{Lemma}

\theoremstyle{definition}

\theoremstyle{remark}

\numberwithin{equation}{section}
\allowdisplaybreaks

\DeclareSymbolFont{largesymbolsCM}{OMX}{cmex}{m}{n}
\DeclareMathSymbol{\intop}{\mathop}{largesymbolsCM}{"52}
\begin{document}

\title{Multiple solutions for a class of semilinear elliptic equations}
\author{Yifu Wang}
\dedicatory{School of Mathematical Sciences, Xiamen University,\\Xiamen 361005, China\\[1ex]{\rm 774550477@qq.com}
}

\begin{abstract}
By using truncation technique, minimization method and Morse theory, we obtain
three nontrivial solutions for a class of semilinear elliptic equations.\bigskip

\medskip


\end{abstract}
\maketitle

\section{Introduction}

In this paper we consider semilinear elliptic boundary value problems of the
form%
\begin{equation}
\left\{
\begin{array}
[c]{ll}%
-\Delta u=g(  u)  \text{,} & \text{in }\Omega\text{,}\\
u=0\text{,} & \text{on }\partial\Omega\text{,}%
\end{array}
\right.  \label{e}%
\end{equation}
where $\Omega$ is a bounded smooth domain in $\mathbb{R}^{N}$, $g$ is a
$C^{1}$-function on $\mathbb{R}$ such that

\begin{enumerate}
\item[$(  \mathbf{g})  $] there are $a^{-}<0<a^{+}$ such that
$g(  a^{\pm})  =0$ and $\delta>0$ such that for some $k\geq2$,%
\[
\lambda_{k}\leq\frac{g(  t)  }{t}\leq\lambda_{k+1}\text{,\qquad for
}t\in(  -\delta,\delta)  \text{.}%
\]

\end{enumerate}

Here, $\lambda_{n}$ denote the $n$th eigenvalue of the Laplacian operator
subject to homogenious Dirichlet boundary condition. In what follows we also
denote by $\phi_{n}$ the eigenfunction corresponding to $\lambda_{n}$ with
$\left\vert \phi_{n}\right\vert _{2}=1$. Note that $(  \mathbf{g})
$ implies that $g(  0)  =0$ and hence $u=0$ is a trivial solution
of \eqref{e}.

\begin{thm}
\label{t1}If $g$ satisfies $(  \mathbf{g})  $, then \eqref{e} has at
least three nontrivial solutions.
\end{thm}

As a special case of the theorem, let $f\in C^{1}(  \mathbb{R})  $
be such that%
\begin{equation}
\lim_{\left\vert u\right\vert \rightarrow0}\frac{f(  t)  }%
{t}=0\text{,\qquad}\lim_{\left\vert u\right\vert \rightarrow\infty}%
\frac{f(  t)  }{t}=+\infty\text{.} \label{e1}%
\end{equation}
We have the following result.

\begin{cor}
Suppose $\lambda>\lambda_{2}$, $\lambda\neq\lambda_{n}$ for all $n$, and
\eqref{e1} holds, then the boundary value problem%
\[
\left\{
\begin{array}
[c]{ll}%
-\Delta u=\lambda u-f(  u)  \text{,} & \text{in }\Omega\text{,}\\
u=0\text{,} & \text{on }\partial\Omega\text{,}%
\end{array}
\right.
\]
has at least three nontrivial solutions.
\end{cor}

\section{The first two nontrivial solutions}

It is well known that variational methods are very useful in studying the
existence and multiplicity of solutions for elliptic boundary value problems.
However, since we have no restriction on the growth rate of the right hand
side of the equation in \eqref{e}, there is not a functional $\Phi:H_{0}%
^{1}(  \Omega)  \rightarrow\mathbb{R}$ whose critical points are
solutions of \eqref{e}, so variational methods are not directly applicable. A
well known method to overcome this is the truncation method.

Set%
\[
g_{+}(t)=\left\{
\begin{array}
[c]{ll}%
g(  t)  \text{,} & \text{if }0\leq t\leq a^{+}\text{,}\\
0\text{,} & \text{if }t<0\text{ or }t>a^{+}%
\end{array}
\right.
\]
and consider the truncated problem%
\begin{equation}
\left\{
\begin{array}
[c]{ll}%
-\Delta u=g_{+}(u)\text{,} & \text{in }\Omega\text{,}\\
u=0\text{,} & \text{on }\partial\Omega\text{.}%
\end{array}
\right.  \label{e2}%
\end{equation}
The following result is well known to experts, since we can not find a
reference containning the proof, we would like to present the detailed proof.

\begin{lem}
\label{l1}Suppose $u\in C^{2}(  \Omega)  \cap C(  \bar{\Omega
})  $ is a classical solution of \eqref{e2}, then%
\begin{equation}
0\leq u(  x)  \leq a^{+}\text{,\qquad for }x\in\Omega\text{.}
\label{e4}%
\end{equation}
Hence $u$ is also a classical solution of \eqref{e}.
\end{lem}

\begin{proof}
Firstly we show%
\[
M=\max_{\bar{\Omega}}u\leq a^{+}\text{.}%
\]
If this is not true, then there is $x_{0}\in\Omega$ such
\[
u(  x_{0})  =M>a^{+}\text{.}%
\]
Let $U=u^{-1}(  M)  $, then $U$ is a closed subset of $\Omega$.

Now pick $x\in U$, then
\[
u(  x)  =M>a^{+}%
\]
So there is $r>0$ such that%
\[
u(y)>a^{+}\text{,\qquad for all }y\in B_{r}(  x)  \text{.}%
\]
Hence $g_{+}(  u)  \equiv0$ in $B_{r}(  x)  $ and
because $u$ is a classical solution of \eqref{e2}, it is harmonic in
$B_{r}(  x)  $. Therefore, since%
\[
u(  x)  =\max_{B_{r}(x)}u\text{,}%
\]
$u$ must be a constant function in $B_{r}(  x)  $. We conclude that
$B_{r}(  x)  \subset U$.

Consequently, $U$ is not only closed but also open in $\Omega$. Hence
$U=\Omega$. This is a contradiction to the condition that $u=0$ on
$\partial\Omega$.

In a similar manner, we can show that%
\[
m=\min_{\bar{\Omega}}u\geq0\text{,}%
\]
and the desired result follows.
\end{proof}

Since $g_{+}$ is bounded, let
\[
G_{+}(  t)  =\int_{0}^{t}g_{+}(  s)  ds\text{,}%
\]
we may define a functional $\Phi_{+}:H_{0}^{1}(\Omega)\rightarrow\mathbb{R}$,%
\[
\Phi_{+}(u)=\frac{1}{2}\int_{\Omega}\left\vert \nabla u\right\vert ^{2}%
dx-\int_{\Omega}G_{+}(  u)  dx\text{.}%
\]
Then $\Phi_{+}$ is of class $C^{1}$ and because we have assumed $g\in
C^{1}(  \mathbb{R})  $, the critical points of $\Phi_{+}$ are
classical solutions of \eqref{e2}.

\begin{lem}
$\Phi_{+}$ is coercive and attains its minimum at some $u_{+}\in H_{0}%
^{1}(  \Omega)  $, which is a nontrivial solution of \eqref{e2}.
\end{lem}

\begin{proof}
Firstly, by $(  \mathbf{g})  $ we see that, there is $\delta>0$,
such that if $u\in(  0,\delta)  $ then%
\[
G_{+}(  u)  \geq\frac{\lambda_{2}}{2}u^{2}\text{.}%
\]
Now for $s\in(  0,\left\vert \phi_{1}\right\vert _{\infty}^{-1}%
\delta)  $, we have%
\[
0<s\phi_{1}(x)<\delta\text{,\qquad for }x\in\Omega\text{.}%
\]
Hence%
\begin{align*}
\Phi_{+}(  s\phi_{1})   &  =\frac{\lambda_{1}s^{2}}{2}\int_{\Omega
}\left\vert \phi_{1}\right\vert ^{2}dx-\int_{\Omega}G_{+}(  s\phi
_{1})  dx\\
&  <\frac{\lambda_{1}s^{2}}{2}\int_{\Omega}\left\vert \phi_{1}\right\vert
^{2}dx-\frac{\lambda_{2}s^{2}}{2}\int_{\Omega}\left\vert \phi_{1}\right\vert
^{2}dx\\
&  =\frac{\lambda_{1}-\lambda_{2}}{2}s^{2}\int_{\Omega}\left\vert \phi
_{1}\right\vert ^{2}dx<0\text{.}%
\end{align*}
Consequently,%
\[
\inf_{H_{0}^{1}(  \Omega)  }\Phi_{+}<0\text{.}%
\]

Obviously, $g_{+}$ is bounded, hence there exist $C_{1}>0$, $C_{2}>0$ such
that%
\[
\left\vert G_{+}(  u)  \right\vert \leq C_{1}+C_{2}\left\vert
u\right\vert \text{.}%
\]
Hence%
\begin{align*}
\Phi_{+}(  u)   &  =\frac{1}{2}\int_{\Omega}\left\vert \nabla
u\right\vert ^{2}dx-\int_{\Omega}G_{+}(  u)  dx\\
&  \geq\frac{1}{2}\left\Vert u\right\Vert ^{2}-C_{1}\left\vert \Omega
\right\vert -C_{2}\left\vert u\right\vert _{1}\\
&  \geq\frac{1}{2}\left\Vert u\right\Vert ^{2}-C_{2}S_{2}\left\Vert
u\right\Vert -C_{1}\left\vert \Omega\right\vert \rightarrow+\infty\text{,}%
\end{align*}
where $S_{p}$ is the Sobolev constant of the embedding $H_{0}^{1}(
\Omega)  \hookrightarrow L^{p}(  \Omega)  $,%
\[
\left\vert u\right\vert _{p}\leq S_{p}\left\Vert u\right\Vert \text{,\qquad
}u\in H_{0}^{1}(  \Omega)  \text{.}%
\]
We see that $\Phi_{+}$ is coercive. It is then well known that there is a
$u_{+}\in H_{0}^{1}(  \Omega)  $ such that%
\[
\Phi_{+}(  u^{+})  =\inf_{H_{0}^{1}(  \Omega)  }\Phi
_{+}<0\text{.}%
\]
Since $\Phi_{+}(  0)  =0$, it follows that $u^{+}$ is a nonzero
critical point of $\Phi_{+}$. Thus $u^{+}$ is a nontrivial solution of
\eqref{e2}.
\end{proof}

Acturally by elliptic regularity, $u^{+}$ is a classical solution of \eqref{e2}.
According to Lemma \ref{l1}, $u^{+}$ satisfies \eqref{e4} and it is a positive
solution of \eqref{e}.

In a similar manner, setting%
\[
g_{-}(t)=\left\{
\begin{array}
[c]{ll}%
g(  t)  \text{,} & \text{if }a^{-}\leq t\leq0\text{,}\\
0\text{,} & \text{if }t>0\text{ or }t<a^{-}%
\end{array}
\right.
\]
and consider%
\[
\left\{
\begin{array}
[c]{ll}%
-\Delta u=g_{-}(u)\text{,} & \text{in }\Omega\text{,}\\
u=0\text{,} & \text{on }\partial\Omega\text{,}%
\end{array}
\right.
\]
we can obtain a negative solution $u^{-}$ of the problem \eqref{e}.

\section{The third nontrivial solution}

To get the third nontrivial solution, we shall apply Morse theory (see
\cite{MR1196690} or \cite[Chapter 8]{MR982267} for a systematic exposition).
The key concept in this theory is critical group.

Let $\Phi$ be a $C^{1}$-functional in a Banach space $X$ and $u$ a critical
point of $\Phi$ with $\Phi(  u)  =c$. We call%
\[
C_{q}(  \Phi,u)  =H_{q}(  \Phi^{c},\Phi^{c}\backslash\left\{
u\right\}  )
\]
the $q$th critical group of $\Phi$ at $u$, where $q=0,1,\ldots$, $H_{q}$
stands for the $q$th singular homology group with coefficient in $\mathbb{Z}$, and
$\Phi^{c}=\Phi^{-1}(-\infty,c]$.

It is known that if $u$ is a critical point produced via the mountain pass
lemma of Ambrosetti and Rabinowitz \cite{MR0370183}, then%
\[
C_{1}(  \Phi,u)  \neq0\text{.}%
\]
To prove our theorem, we need to truncate the problem once more. Let%
\[
\tilde{g}(t)=\left\{
\begin{array}
[c]{ll}%
g(  t)  \text{,} & \text{if }a^{-}\leq t\leq a^{+}\text{,}\\
0\text{,} & \text{if }t\notin\lbrack a^{-},a^{+}]
\end{array}
\right.
\]
and consider%
\begin{equation}
\left\{
\begin{array}
[c]{ll}%
-\Delta u=\tilde{g}(u)\text{,} & \text{in }\Omega\text{,}\\
u=0\text{,} & \text{on }\partial\Omega\text{.}%
\end{array}
\right.  \label{e5}%
\end{equation}
The solutions of \eqref{e5} are critical points of the $C^{2-0}$-functional $\Phi:H_{0}^{1}(
\Omega)  \rightarrow\mathbb{R}$,%
\[
\Phi(  u)  =\frac{1}{2}\int_{\Omega}\left\vert \nabla u\right\vert
^{2}dx-\int_{\Omega}G(  u)  \text{,\qquad}G(  u)
=\int_{0}^{u}\tilde{g}(  s)  ds\text{.}%
\]
Similar to Lemma \ref{l1}, we have the following result.

\begin{lem}
Suppose $u\in C^{2}(  \Omega)  \cap C(  \bar{\Omega})  $
is a classical solution of \eqref{e5}, then%
\begin{equation}
a^{-}\leq u(  x)  \leq a^{+}\text{,\qquad for }x\in\Omega\text{.}
\label{e6}%
\end{equation}
Hence $u$ is also a classical solution of \eqref{e}.
\end{lem}

Consider the solution $u^{+}$ of \eqref{e} obtained in the last section: it is a
global minimizer of $\Phi_{+}$ on $H_{0}^{1}(  \Omega)  $. By
strong maximum principle we know that%
\begin{equation}
u^{+}>0\text{\quad in }\Omega\text{,\qquad}\frac{\partial u^{+}}{\partial\nu
}>0\text{\quad on }\partial\Omega\text{,}\label{ee}%
\end{equation}
where $\nu$ is the interior normal on $\partial\Omega$. From this it follows
that $u^{+}$ is an interior point of the set%
\[
P=\left\{  \left.  u\in C^{1}(  \bar{\Omega})  \right\vert
\,u>0\text{ in }\Omega\text{, }u=0\text{ on }\partial\Omega\right\}
\]
with respect to the $C^{1}$-topology. In fact, if this is not true, we may
find $\left\{  v_{n}\right\}  \subset C^{1}(  \bar{\Omega})  $ and
$\left\{  x_{n}\right\}  \subset\Omega$ such that $v_{n}(  x_{n})
<0$ and%
\begin{equation}
\sup_{\Omega}\left\vert v_{n}-u^{+}\right\vert +\sup_{\Omega}\left\vert \nabla
v_{n}-\nabla u^{+}\right\vert \rightarrow0\text{.}\label{e7}%
\end{equation}
Since $\Omega$ is bounded, we may assume $x_{n}\rightarrow x^{\ast}\in
\bar{\Omega}$. Because%
\begin{align*}
\left\vert v_{n}(  x_{n})  -u^{+}(  x^{\ast})
\right\vert  &  \leq\left\vert v_{n}(  x_{n})  -u^{+}(
x_{n})  \right\vert +\left\vert u^{+}(  x_{n})  -u^{+}(
x^{\ast})  \right\vert \\
&  \leq\sup_{\Omega}\left\vert u^{+}-v_{n}\right\vert +\left\vert u^{+}(
x_{n})  -u^{+}(  x^{\ast})  \right\vert \rightarrow0\text{,}%
\end{align*}
we see that%
\[
u^{+}(  x^{\ast})  =\lim_{n\rightarrow\infty}v_{n}(
x_{n})  \leq0\text{.}%
\]
Because $u^{+}>0$ in $\Omega$, we deduce $x^{\ast}\in\partial\Omega$.
Similarly, using \eqref{e7} again we have%
\[
\nabla u^{+}(  x^{\ast})  =\lim_{n\rightarrow\infty}\nabla
v_{n}(  x_{n})  =0\text{,}%
\]
Thus%
\[
\left.  \frac{\partial u^{+}}{\partial\nu}\right\vert _{x^{\ast}}=\nabla
u^{+}(  x^{\ast})  \cdot\nu=0\text{,}%
\]
a contradiction to \eqref{ee}.

Thus, there exists $r>0$ such that if $u\in C^{1}(  \bar{\Omega})
$, $u=0$ on $\partial\Omega$ and%
\[
\left\Vert u-u^{+}\right\Vert _{C^{1}}<r\text{,}%
\]
then $u>0$ in $\Omega$. Consequently, $\Phi(  u)  =\Phi_{+}(
u)  $. Therefore, $u^{+}$ is also a local minimizer of $\Phi$ in $C^{1}$-topology.

By a result of Br\'{e}zis and Nirenberg \cite{MR1239032}, we conclude that
$u^{+}$ is a local minimizer of $\Phi$ in the $H^{1}$-topology.

Similarly, the negative solution $u^{-}$ of \eqref{e} is also a local minimizer
of $\Phi$. Because $\Phi$ satisfies the Palais-Smale condition, from the two
local minimizer $u^{+}$ and $u^{-}$ we can obtain a third critical point
$u^{\ast}$ of $\Phi:H_{0}^{1}(  \Omega)  \rightarrow\mathbb{R}$ via
the mountain pass lemma. Since $u^{\ast}$ is of mountain pass type, we have%
\begin{equation}
C_{1}(  \Phi,u^{\ast})  \neq0\text{.}\label{e88}%
\end{equation}
On the other hand, by our assumption $(  \mathbf{g})  $ and
\cite[Proposition 1.1]{MR1838509}, we have%
\begin{equation}
C_{q}(  \Phi,0)  =\left\{
\begin{array}
[c]{ll}%
\mathbb{Z}\text{,} & \text{if }q=k\text{,}\\
0\text{,} & \text{if }q\neq k\text{.}%
\end{array}
\right.  \label{e9}%
\end{equation}
Since $k\geq2$, From \eqref{e88} and \eqref{e9} we see that%
\[
C_{1}(  \Phi,u^{\ast})  \neq C_{1}(  \Phi,0)  \text{.}%
\]
Hence, $u^{\ast}\neq0$, which is the third nontrivial solution of \eqref{e}.

\end{document}